\documentclass[12pt]{article}
\usepackage{graphicx,color}
\usepackage{amsmath,amssymb,amsthm}

\newcommand{\rr}[2]{$\displaystyle{\phantom{\bigg|}\!\!\!{#1\atop (#2\%)}\phantom{\bigg|}\!\!\!}$}

\newtheorem*{VDW}{van der Waerden's Theorem}
\newtheorem{theorem}{Theorem}
\newtheorem{corollary}[theorem]{Corollary}
\newtheorem{lemma}[theorem]{Lemma}

\newtheorem{question}[theorem]{Question}

\newtheorem{conjecture}[theorem]{Conjecture}

\begin{document}
\title{Unrolling residues to avoid progressions}
\author{Steve Butler\footnote{Department of Mathematics, Iowa State University, Ames, IA  50011, USA ({\tt butler@iastate.edu})} \and Ron Graham\footnote{Department of Computer Science and Engineering, University of California San Diego, La Jolla, CA 92093, USA ({\tt graham@ucsd.edu})} \and Linyuan Lu\footnote{Department of Mathematics, University of South Carolina, Columbia, SC 29208, USA ({\tt lu@math.sc.edu})}}
\date{\empty}
\maketitle

\begin{abstract}
We consider the problem of coloring $[n]=\{1,2,\ldots,n\}$ with $r$ colors to minimize the number of monochromatic $k$ term arithmetic progressions (or $k$-APs for short).  We show how to extend colorings of $\mathbb{Z}_m$ which avoid nontrivial $k$-APs to colorings of $[n]$ by an unrolling process.  In particular, by using residues to color $\mathbb{Z}_m$ we produce the best known colorings for minimizing the number of monochromatic $k$-APs for coloring with $r$ colors for several small values of $r$ and $k$.
\end{abstract}

Ramsey Theory is sometimes described by the maxim that complete disorder is impossible.  A well known example of this phenomenon is that among any six people, three of them must be mutual friends or three of them must be complete strangers.  If we increase the number of people to eighteen then there will be either four mutual friends or four complete strangers.  In these cases the problem relates to colorings of the edges of the complete graph with two colors and showing that there must be some small clique which is monochromatic.

We are going to focus on a Ramsey Theory problem related to coloring $[n]=\{1,2,\ldots,n\}$.  Instead of looking for a clique which is monochromatic now we are interested in finding a short arithmetic progression which is monochromatic.  An arithmetic progression of length $k$, sometimes abbreviated as a $k$-AP, is a sequence of $k$ numbers of the form
\[
a,a+d,a+2d,\ldots,a+(k-1)d.
\]
We say that a $k$-AP is monochromatic if all of the numbers in the sequence were assigned the same color.

\begin{VDW}[see \cite{LR,vdw}]
For any $r$ and $k$ there is an $N$ so that if $n\ge N$ then for \emph{any} coloring of $[n]$ using $r$ colors there must be an arithmetic progression of length $k$ which is monochromatic.
\end{VDW}

This shows that there must be at least one monochromatic arithmetic progression when our $n$ is sufficiently large.  Frankl, Graham and R\"odl \cite{FGR} were able to extend this result to establish an even stronger statement, namely that there must be many.

\begin{corollary}\label{dense}
For any $r$ and $k$ there is a $\gamma >0$ so that the number of monochromatic $k$-APs in any $r$-coloring of $[n]$ is at least $\big(\gamma+o(1)\big)n^2$.
\end{corollary}

[In this corollary and throughout the paper we will use little-``$o$'' and big-``O'' notation which are used to bound the size of the lower order terms.  This is a convenient tool which allows us to focus on the leading terms.  An introduction to this notation can be found in many places including Graham, Knuth and Patashnik \cite{GKP}.]

To put the result of Corollary~\ref{dense} in perspective, the total number of $k$-APs in $[n]$ is ${1\over 2(k-1)}n^2+O(n)$ showing that a positive proportion of these arithmetic progressions must be monochromatic.  One way to see the count for the number of $k$-APs is to note that there is a $k$-AP between any two points if and only if the difference between them is a multiple of $k-1$.  There are ${n(n+1)\over 2}$ ways to pick two points and with probability ${1\over k-1}$ the difference is divisible by $k-1$.  The $O(n)$ term handles any rounding errors that might occur due to the value of $n\pmod{k-1}$.  

This leads to the following open question.

\begin{question}
Given $k$ and $r$ what is the largest $\gamma$ so that for \emph{any} coloring of $[n]$ there are at least $\big(\gamma+o(1)\big)n^2$ monochromatic $k$-APs?  Or, equivalently, what is the smallest $\gamma$ so that for \emph{some} coloring of $[n]$ there are at most $\big(\gamma+o(1)\big)n^2$ monochromatic $k$-APs?
\end{question}

Our goal in this paper is to partially address this question for small values of $r$ and $k$ by exhibiting colorings of $[n]$ which have a small value of $\gamma$.  A natural first guess for the value of $\gamma$ is to color ``randomly''.  Or in other words we pick a typical coloring and consider how many monochromatic $k$-APs should the coloring be expected to contain.  This bound can be easily computed by counting.

\begin{lemma}\label{lem:random}
There is a coloring of $[n]$ with $r$ colors which has at most ${1\over 2(k-1)r^{k-1}}n^2+O(n)$ monochromatic $k$-APs.
\end{lemma}
\begin{proof}
To establish this we will simply count the total number of monochromatic $k$-APs over all colorings and then take the average (i.e., divide by $r^n$, the number of colorings).  Some coloring must have at most the average number which will then give us the bound.

Fix the location inside of $[n]$ of a $k$-AP.  There are $r$ ways to color that $k$-AP monochromatically and we can extend each such coloring to a coloring of $[n]$ by $r^{n-k}$ ways.  Therefore each fixed $k$-AP will contribute $r^{n-k+1}$ to the sum total number of monochromatic $k$-APs over all possible colorings.

On the other hand we have already counted the number of $k$-APs in $[n]$ so that the total number of monochromatic $k$-APs over all possible colorings is
\[
\bigg({1\over 2(k-1)}n^2+O(n)\bigg)r^{n-k+1}.
\]
Dividing this by $r^n$ to get the average number of $k$-APs now establishes the result.
\end{proof}

The random bound was the best known bound for any $r$ and $k$ until several years ago when two groups independently found the same way to color $[n]$ that was an improvement in the case of $2$ coloring $[n]$ to avoid $3$-APs (where the random bound gives ${1\over 16}n^2+O(n)$).

\begin{theorem}[Parrilo-Robertson-Saracino \cite{PRS}; Butler-Costello-Graham \cite{BCG}]
Divide $[n]$ into twelve blocks with relative sizes 
\[
28,6,28,37,59,116,116,59,37,28,6,28
\]
and color each block a solid color where we alternate the colors between the blocks.  Then the number of monochromatic three term arithmetic progressions in such a coloring is
\[
{117\over 2192}n^2+O(n)={117\over137}{\cdot}{1\over 16}n^2+O(n).
\]
\end{theorem}

An example of the resulting coloring, where we use red and blue, is shown in Figure~\ref{fig:3ap}.

\begin{figure}
\centering
\includegraphics{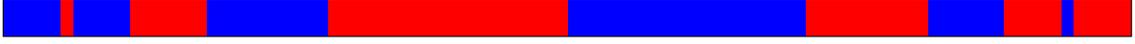}
\caption{The best known way to color $[n]$ with $2$ colors to avoid $3$-APs.}
\label{fig:3ap}
\end{figure}

\begin{conjecture}
Each coloring of $[n]$ with two colors has at least ${117\over 2192}n^2+O(n)$ monochromatic three term arithmetic progressions.
\end{conjecture}

Butler, Costello and Graham \cite{BCG} used a computer search to look for similar ways to color $[n]$ by subdividing into blocks to avoid $4$-APs and $5$-APs.  For avoiding $4$-APs they found a block coloring with $36$ blocks that had $0.017221n^2+O(n)$ monochromatic $4$-APs; and a block coloring with $113$ blocks that had $0.005719n^2+O(n)$ monochromatic $5$-APs.

After this first sign of progress on the problem in thirty years, it only took six months for a better alternative to block coloring to be discovered that consisted of unrolling a coloring of $\mathbb{Z}_m$.

\subsection*{Unrolling colorings of $\mathbb Z_m$}
In addition to the problem of minimizing the number of monochromatic $k$-APs in $[n]$, one can also look at minimizing the number of monochromatic $k$-APs in $\mathbb Z_n$ (the integers modulo $n$).  In this setting the problem of avoiding $3$-APs has a surprising result.  Namely, Cameron, Cilleruelo and Serra \cite {CCS} showed that for avoiding $3$-APs when two colors are involved, it does not matter how you color the elements, only how much of each color is used.

For $4$-APs the way that $\mathbb{Z}_n$ is colored makes a significant difference.  The current best result for $4$-APs was found by Lu and Peng \cite{LP} who repeatedly copied the coloring shown in Figure~\ref{fig:Z11} of $\mathbb{Z}_{11}$ to color $\mathbb{Z}_n$.  This coloring of $\mathbb{Z}_{11}$ has no nontrivial monochromatic $4$-AP regardless of how the red bit in Figure~\ref{fig:Z11} is colored (an example of a trivial $4$-AP is something of the form $a,a,a,a$).  Because there are no nontrivial $4$-APs, this forced any $4$-AP in their coloring of $\mathbb{Z}_n$ to have large gaps, and hence few monochromatic $4$-APs.

\begin{figure}[h]
\centering
\includegraphics[scale=1]{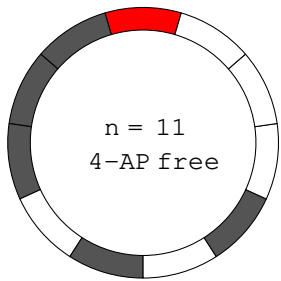}
\caption{A coloring of $Z_{11}$ with no nontrivial monochromatic $4$-AP; the red bit can be either color.}
\label{fig:Z11}
\end{figure}

Lu and Peng then observed that this same approach of coloring to force the $4$-APs to have large gaps in $\mathbb{Z}_n$ could also be done for $[n]$.  This is done by \emph{unrolling} the coloring of $\mathbb Z_{11}$, i.e., by placing repeated copies of the coloring of $\mathbb{Z}_{11}$ (see Figure~\ref{fig:unroll}).  Any monochromatic $4$-AP in $[n]$ would correspond to a monochromatic $4$-AP in $\mathbb{Z}_{11}$ which must be trivial, forcing the monochromatic $4$-APs to take steps that are multiples of $11$.

\begin{figure}[b]
\centering
\includegraphics[scale=0.85]{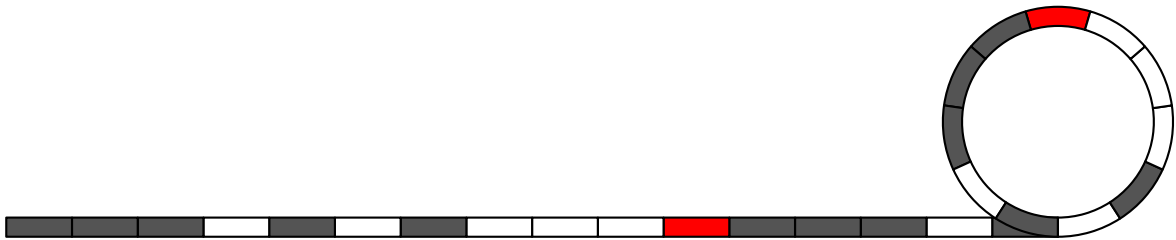}
\caption{Unrolling the coloring of $Z_{11}$ to find a coloring of $[n]$}
\label{fig:unroll}
\end{figure}

There is further improvement in that we still have freedom to color the red squares that showed up in the unrolling arbitrarily.  To do this we simply recurse and unroll along the red squares the same pattern we did initially.  This is much easier to do than at first glance, namely the coloring is equivalent to doing the following:  Given $\ell$  write it in base $11$, i.e., 
\[
\ell=\sum_ib_i11^i\qquad\mbox{where }0\le b_i\le 10,
\]
then if $j$ is the smallest index so that $b_j\ne 0$,
\[
\mbox{color $\ell$ }\left\{
\begin{array}{l@{\quad}l}
\mbox{white}&\mbox{if }b_j\in\{1,3,4,5,9\};\\
\mbox{black}&\mbox{if }b_j\in\{2,6,7,8,10\}.
\end{array}\right.
\]

It will follow from Theorem~\ref{thm:unroll} below that this coloring has
\[
{1\over 72}n^2+O(n)\approx 0.013888n^2 < 0.020833n^2\approx{1\over 48}n^2+O(n).
\]
This coloring is far superior to the block coloring referred to earlier, and also far easier to describe and to extend (indeed we can color all of $\mathbb{N}$ and then just take the first $n$ terms with this coloring; in a block coloring such as in Figure~\ref{fig:3ap} this would be impossible).  

This idea of unrolling works in general for any coloring of $\mathbb{Z}_m$ which avoids nontrivial $k$-APs.  Further, counting the number of monochromatic $k$-APs can be easily done.

\begin{theorem}\label{thm:unroll}
If there is an $r$ coloring of $\mathbb Z_m$ which avoids nontrivial $k$-APs, then there is an $r$ coloring of $[n]$ which has 
\[
{1\over 2m(k-1)}n^2+O(n)
\]
monochromatic $k$-APs.

If there are $r$ different $r$ colorings of $Z_m$ which avoid nontrivial $k$-APs and these differ only in the coloring of $0$, then there is an $r$ coloring of $[n]$ which has 
\[
{1\over 2(m+1)(k-1)}n^2+O(n)
\]
monochromatic $k$-APs.
\end{theorem}
\begin{proof}
As noted above, in both cases all of the resulting monochromatic arithmetic progressions must lie in the same residue classes mod $m$.

In the first case, we color according to the residue class in the coloring of $Z_m$.  In particular the residue completely determines the color and so we can simply count our coloring by counting what happens on each of the $m$ residue classes.  Each residue consists of a solid coloring of length ${n\over m}+O(1)$ and we already know how to count the number of progressions in solid colorings.  This establishes the first case.

In the second case we unroll the coloring.  Namely for $i=1,2,\ldots,r$ let $C_i$ be the nonzero elements of $\mathbb Z_m$ colored with color $i$.  For each $\ell$ we find the first nonzero digit in the base $m$ expansion of $\ell$ and color $\ell$ the $i$th color if the digit is in $C_i$.   Initially when we subdivide by residue classes then as before $m-1$ of the residue classes will be colored solid.  The residue class for $0$ will look like the unrolling for the coloring of ${n\over m}+O(1)$.

If we let $F(n)$ be the number of monochromatic $k$ term arithmetic progressions then the above analysis shows that
\[
F(n)=F\bigg({n\over m}\bigg)+(m-1){1\over 2(k-1)}\bigg({n\over m}\bigg)^2+O(n).
\]
Iteratively applying this we have
\begin{align*}
F(n)&={(m-1)n^2\over 2(k-1)}\sum_{i\ge 1}\bigg({1\over m^2}\bigg)^i+O(n)\\
&= {(m-1)n^2\over 2(k-1)}\cdot{1\over m^2-1}+O(n)\\
&= {1\over 2(m+1)(k-1)}n^2+O(n). \qedhere
\end{align*}
\end{proof}

\subsection*{Using residues to color}
To apply Theorem~\ref{thm:unroll} we must find a coloring of $\mathbb{Z}_m$ that only contains trivial arithmetic progressions of length $k$.  Furthermore, given a $k$ we want $m$ to be as large as possible.

A close examination of the coloring of $\mathbb{Z}_{11}$ reveals that the nonzero elements colored white, $\{1,3,4,5,9\}$, are precisely the nonzero quadratic residues of $\mathbb{Z}_{11}$, i.e., values $r\in\mathbb{Z}_{11}$ for which there is some $x$ satisfying $x^2\equiv r\pmod{11}$.  

This suggests that we can look at colorings of $\mathbb{Z}_p$ where $p$ is a prime and where we color the quadratic residues white and the non-residues black.  This has two large advantages.
\begin{itemize}
\item Suppose that $a,a+d,\ldots,a+(k-1)d$ is a nontrivial monochromatic $k$-AP, i.e., $d\not\equiv 0\pmod p$.  Then we can multiply each term by $d^{-1}$ to get the arithmetic progression $ad^{-1}, ad^{-1}+1,\ldots,ad^{-1}+(k-1)$ consisting of $k$ consecutive elements of $\mathbb{Z}_p$.  Further since we started with a sequence which consisted either of only residues or  non-residues the resulting sequence must also consist of only residues or non-residues.  (This is because if $d^{-1}$ is a quadratic residue then multiplication by $d^{-1}$ sends residues to residues and non-residues to non-residues; a similar situation occurs if $d^{-1}$ is not a quadratic residue.)

This shows that to determine the length of the longest nontrivial arithmetic progression in a coloring by quadratic residues it suffices to find the longest run of residues or non-residues.  This is a tremendous speedup and makes large computer searches possible.

\item It is known that for $\mathbb{Z}_p$ that there are no long runs of consecutive residues or non-residues.  In particular, Burgess \cite{Burg} showed that the maximum number of consecutive quadratic residues or non-residues for a prime $p$ is $O\big(p^{1/4}(\log p)^{3/2}\big)$.

(The result of Burgess is not strong enough to guarantee the type of colorings we need for large $k$.  If we compare the random bound with Theorem~\ref{thm:unroll} then we need the length of the longest monochromatic progression in $\mathbb{Z}_p$ to be $\le \log_2(p)$ to yield a coloring which is better than random for some $k$.  On the other hand the result of Burgess is a general bound and does not preclude the possibility that for some (possibly many) primes $p$ the length might be smaller than $\log_2(p)$.)
\end{itemize}

When we want to use more than two colors we can use higher order residues in place of quadratic residues to look for colorings which avoid nontrivial $k$-APs.  In general, suppose $p$ is a prime, so that $\mathbb{Z}_p^*$ (the invertible elements of $\mathbb{Z}_p$) is a group of order $p-1$.  Further if $r\big|(p-1)$ then there is a unique subgroup of index $r$, namely
\[
S=\{x^r\mid x\in\mathbb{Z}_p,x\ne0\}.
\]
If there is no nontrivial $k$-AP in $S\cup\{0\}$ then there cannot be a nontrivial $k$-AP in $yS\cup\{0\}$ for any $y\in\mathbb{Z}_p^*$ (i.e., if there were, multiply the progression by $y^{-1}$ to find one in $S\cup\{0\}$, a contradiction).  In particular by choosing $y_1=1,y_2,\ldots,y_r$ so that the collection of $y_iS$ form the cosets of $\mathbb{Z}_p^*/S$ we form a coloring of $\mathbb{Z}_p$ with $r$ colors which has no $k$-AP regardless of the color assignment of $0$.  We have now established the following.

\begin{theorem}\label{tr}
Suppose that $p$ is an odd prime and $r\big|(p-1)$.  If $\{x^r:x\in \mathbb{Z}_p\}$ contains no nontrivial $k$-AP then there is a coloring with $r$ colors where $0$ can be colored arbitrarily and which contains no nontrivial $k$-APs.
\end{theorem}

In some cases we can combine two good colorings to form a larger coloring, as shown in the following theorem.

\begin{theorem}\label{tensor}
For $i=1,2$, let $\mathcal{C}_i$ be a coloring of $\mathbb{Z}_{m_i}$ using $r_i$ colors where $0$ can be colored arbitrarily and containing no nontrivial $k$-APs.  Then there exists a coloring $\mathcal{C}$ of $\mathbb{Z}_{m_1m_2}$ using $r_1r_2$ colors where $0$ can be colored arbitrarily and containing no nontrivial $k$-APs.
\end{theorem}
\begin{proof}
We show how to color $\mathbb{Z}_{m_1m_2}$ using the colors $(c_1,c_2)$ where $c_1$ is a color from $\mathcal{C}_1$ and $c_2$ is a color from $\mathcal{C}_2$.  Given $0\le z<m_1m_2$ it can uniquely be written as $x m_2+y$ where $0\le x<m_1$ and $0\le y<m_2$.  This gives a map from $\mathbb{Z}_{m_1m_2}\to \mathbb{Z}_{m_1}\times\mathbb{Z}_{m_2}$ with the following property: If $z_1\mapsto(x_1,y_1)$ and $z_2\mapsto(x_2,y_2)$ then $z_1+z_2\mapsto(x',y_1+y_2\pmod{m_2})$.

The coloring is now defined by the following:  Given $z\in\mathbb{Z}_{m_1m_2}$ we color it as $\big(c_1(x),c_2(y)\big)$ where $c_1(x)$ is the color of $x$ in $\mathcal{C}_1$ and $c_2(y)$ is the color of $y$ in $\mathcal{C}_2$.

It remains to show that this has no nontrivial monochromatic $k$-APs.  Suppose that $z_1,z_2,\ldots,z_k$ were a monochromatic $k$-AP where $z_i\mapsto(x_i,y_i)$.  Then $y_1,y_2,\ldots,y_k$ is a monochromatic $k$-AP in $\mathcal{C}_2$.  But the only way this is possible is if $y_1=y_2=\dots=y_k$.

We can think of a $k$-AP as found by sewing together several $3$-APs.  So we have for $2\le i\le k-1$ (with what we have about the equality of the $y_i$)
\[
0\equiv z_{i-1}-2z_i+z_{i+1}\equiv m_2(x_{i-1}-2x_i+x_{i+1})\pmod{m_1m_2};
\]
which implies
\[
x_{i-1}+x_{i+1}\equiv 2x_i\pmod{m_1}.
\]
This shows that we now have a sequence of $3$-APs glued together in $\mathbb{Z}_{m_1}$, i.e., $x_1,x_2,\ldots,x_k$ is a monochromatic $k$-AP in $\mathbb{Z}_{m_1}$.  But the only way this is possible is if $x_1=x_2\dots=x_k$.

We now conclude that $z_1=z_2=\cdots=z_k$, i.e., there are no nontrivial monochromatic $k$-APs.  Finally we note that our argument does not depend on how we choose to color $0$.
\end{proof}

By computer search we can now find several colorings by either using residues (as in Theorem~\ref{tr}) or combinations of colorings from residues (as in Theorem~\ref{tensor}).  The results of the search are shown in Table~\ref{tab:2}.  We indicate both the $m$ (those colored red are using Theorem~\ref{tensor}) and also compare how well this coloring does compared to the random bound (see Lemma~\ref{lem:random}) expressed as a percentage.  Since each of these colorings allows $0$ to be arbitrary we will use the stronger form of Theorem~\ref{thm:unroll} so that the comparison with random can be found by computing ${r^{k-1}\over m+1}$.

\begin{table}
\centering
\begin{tabular}{|c||c|c|c|c|c|c|}\hline
\rr{m}{} & $r=2$ & $r=3$ &$r=4$ &$r=5$ & $r=6$ & $r=7$\\ \hline \hline
$k=3$& & & \rr{37}{42.11}&  &\rr{103}{34.95}&\\ \hline
$k=4$& \rr{11}{66.67}&\rr{97}{27.55}&\rr{349}{18.29}&\rr{751}{16.62}&\rr{3259}{6.63}&\rr{1933}{17.74}\\ \hline
$k=5$&\rr{37}{42.11}&\rr{241}{33.47}&\rr{2609}{9.81}&\rr{6011}{10.40}&\rr{14173}{9.14}&\rr{30493}{7.87}\\ \hline
$k=6$&\rr{139}{22.85}&\rr{1777}{13.67}&\rr{\textcolor{red}{139^2}}{5.30}&\rr{49391}{6.32}&\rr{\textcolor{red}{139\cdot1777}}{3.15}&\rr{317969}{5.29}\\ \hline
$k=7$&\rr{617}{10.36}&\rr{7309}{9.87}&\rr{\textcolor{red}{617^2}}{1.08}&\rr{230281}{6.78}&\rr{\textcolor{red}{617\cdot7309}}{1.03}& \\ \hline
$k=8$&\rr{1069}{11.96}&\rr{34057}{6.42}&\rr{\textcolor{red}{1069^2}}{1.43}&&&\\ \hline
$k=9$&\rr{3389}{7.55}&\rr{116593}{5.63}&&&&\\ \hline
$k=10$&\rr{11497}{4.45}&\rr{463747}{4.24}&&&&\\ \hline
$k=11$&\rr{17863}{5.73}&&&&&\\ \hline
$k=12$&\rr{58013}{3.53}&&&&&\\ \hline
$k=13$&\rr{136859}{2.99}&&&&&\\ \hline
$k=14$&\rr{239873}{3.41}&&&&&\\ \hline
$k=15$&\rr{608789}{2.69}&&&&&\\ \hline
$k=16$&\rr{1091339}{3.00}&&&&&\\ \hline
\end{tabular}
\caption{Best known $m$ for which there exists a coloring of $Z_m$ with $r$ colors where $0$ can be colored arbitrarily and containing no $k$-APs (also the percentage of the resulting coloring versus random colorings).}
\label{tab:2}
\end{table}

\subsection*{Open problems and variations}
By using residues we have found colorings of $\mathbb{Z}_m$ which when unrolled provide ways to color $[n]$ which are superior to random colorings.  However there is no reason why we must  use residues for our colorings.  As an example if we work with $r=3$ then $-1$, $0$ and $1$ are always cubic residues and so we cannot avoid $3$-APs.  On the other hand there are colorings of $\mathbb{Z}_{12}$ using three colors (shown in Figure~\ref{fig:threecolors}) which by Theorem~\ref{thm:unroll} when unrolled have ${1\over48}n^2+O(n)$ monochromatic $3$-APs.  These colorings have $75\%$ of what we would expect for a random coloring.  We still have a lot to learn about ways to color $\mathbb{Z}_m$ beyond residues which give efficient unrollings.

\begin{figure}
\centering
\includegraphics[scale=0.8]{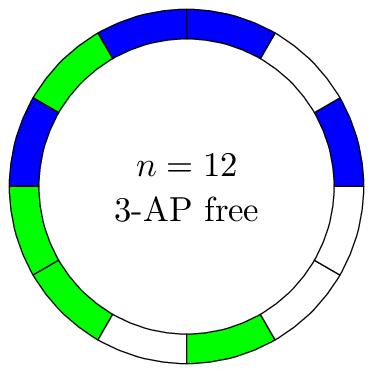}\hfil
\includegraphics[scale=0.8]{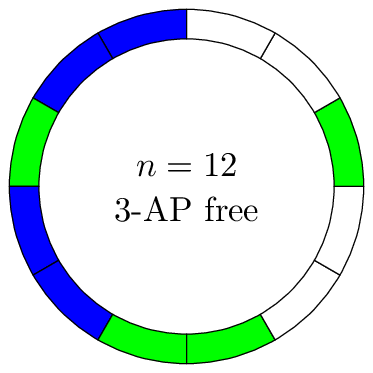}\hfil
\includegraphics[scale=0.8]{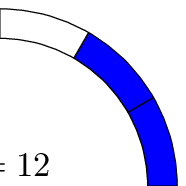}\hfil
\includegraphics[scale=0.8]{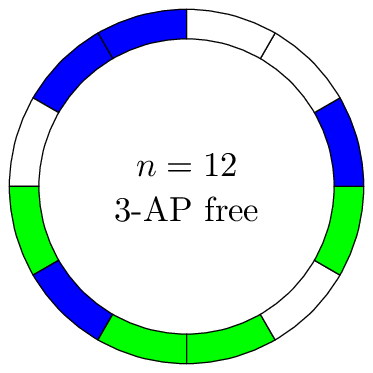}\hfil
\includegraphics[scale=0.8]{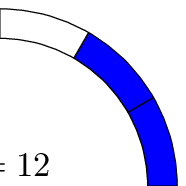}
\caption{Colorings of $Z_{12}$ with no nontrivial monochromatic $3$-AP using three colors.}
\label{fig:threecolors}
\end{figure}

There remains a large amount of work to be done in answering the original question of determining the largest $\gamma$ so that for any coloring using $r$ colors there will be $\big(\gamma+o(1)\big)n^2$ monochromatic $k$-APs.  We have mostly focused on finding colorings which beat random, which give upper bounds for $\gamma$ but we have not even touched on the subject of lower bounds or attempted to show that any of the colorings that we have found are optimal.  For the simplest nontrivial case when $r=2$ and $k=3$ the best known lower bound is due to Parrilo, Robertson and Saracino \cite{PRS} who showed that $\gamma$ lies in the interval
\[
{1675\over 32768}n^2+o(n^2)\approx 0.05111n^2<0.05338n^2\approx {117\over 2192}n^2+O(n).
\]

In general we have provided some evidence in support of showing that we can always beat the random bound for minimizing the number of monochromatic $k$-APs.  While this evidence is compelling we must be careful not to be misled by these small cases.  We will need a more general machinery to establish the following.

\begin{conjecture}
For each $k\ge 3$ and $r\ge 2$ there is a coloring of $[n]$ with $r$ colors so that there are $\delta n^2+O(n)$ monochromatic progressions of length $k$ and where $\delta<{1\over 2(k-1)r^{k-1}}$.
\end{conjecture}

Besides considering ways to avoid $k$-APs one can also try to avoid other patterns.  For example, we can try to avoid $a,a+2d,a+3d,a+5d$ (which can be thought of as a $6$-AP with some terms punched out).  In this case the best known coloring with two colors for unrolling is found by a coloring of $\mathbb{Z}_{13}$ and is shown in Figure~\ref{fig:13} (as before the red bit can be arbitrary).  Similar calculations to Theorem~\ref{thm:unroll} will show that this unrolling has ${1\over140}n^2+O(n)$ monochromatic patterns (about $57.14\%$ of what we would expect of a random coloring).

\begin{figure}[h]
\centering
\includegraphics[scale=1]{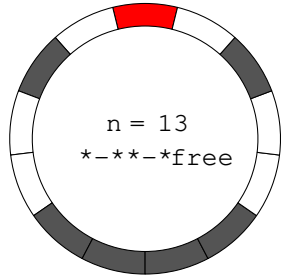}
\caption{A coloring of $Z_{13}$ with no nontrivial monochromatic {\tt*-**-*}; the red bit can be either color.}
\label{fig:13}
\end{figure}

There are still many interesting problems rolled up in this corner of combinatorics and we look forward to seeing more work unroll in this direction.

\eject

\end{document}